\documentclass[english]{smfart}

\usepackage[colorlinks=true,linkcolor=black,citecolor=black,urlcolor=black]{hyperref}
\usepackage{amsmath,amsfonts,mathrsfs,amstext, amsthm, amssymb}
\usepackage{graphicx}

\newtheorem{theorem}{Theorem}[section]
\newtheorem{theo}[theorem]{Theorem}

\newtheorem{coro}[theorem]{Corollary}
\newtheorem{prop}[theorem]{Proposition}
\newtheorem*{theor}{\it\bfseries Theorem}
\newtheorem*{corol}{\it\bfseries Corollary}
\newtheorem*{theoc}{\it\bfseries Theorem B.2.1}
\newtheorem*{coroc}{\it\bfseries Corollary B.2.2}

\def\qbar{{\bar{\bold Q}}}

\def\fpbar{{\bar{\bold F}_p}}

\title[The supersingular elliptic curve in characteristic 2]{A singular property of the supersingular elliptic curve in characteristic $2$}
\author{Leonardo Zapponi}
\email{zapponi@math.jussieu.fr}

\date{\today}

\begin{document}

\dedicatory{Dedicated to J\"urgen Wolfart for his 65th birthday}

\maketitle

\begin{abstract}{\footnotesize Let $E$ be the supersingular elliptic curve defined over the field $k=\bar{\bold F}_2$, which is unique up to $k$-isomorphism. Denote by $0_E$ its identity element and let $C\cong\bold A^1_k$ the quotient of $E-\{0_E\}$ under the action of the group $\mbox{Isom}_k(E)$ (which is non-abelian, of order $24$). The main result of this paper asserts that the set $C(k)$ naturally parametrizes $k$-isomorphism classes of Lam\'e covers, which are tamely ramified covers $X\to\bold P^1$ unramified outside three points having a particular ramification datum. This fact is surprising for two reasons: first of all, it is the first non-trivial example of a family of covers of the projective line unramified outside three points which is parametrized by the geometric points of a curve. Moreover, when considered in arbitrary characteristic, the explicit construction of Lam\'e covers is quite involved and their arithmetic properties still remain misterious. For example, in characteristic $0$, they are closely related to Lam\'e operators with finite dihedral monodromy and their study leads to particular modular forms for the congruence subgroups $\Gamma_1(n)$. The simplicity of the problem in characteristic $2$ has many deep consequences when combined with lifting techniques from positive characteristic to characteristic $0$. As an illustration, it turns out that for any odd integer $n$, there is a bijection between the set of $\bold C$-isomorphism classes of Lam\'e operators with dihedral monodromy of order $2n$ and the quotient of the $n$-torsion subgroup $E[n]$ under the action of $\mbox{Isom}_k(E)$. We also obtain some sharp statements concerning the (local) Galois action on Lam\'e covers, as, for example, a bound on the number of isomorphism classes of Lam\'e covers defined over a fixed number field, only depending on the degree of the residual extension at $2$. Finally, in the appendix we give a partial generalization of these results by showing that, for any positive integer $g$, the $k$-rational points of a suitable quotient of a genus $g$ hyperelliptic curve parametrize $k$-isomorphism classes of tamely ramified covers of the projective line unramified outside three points.}
\end{abstract}

\newpage

\tableofcontents

\newpage

\section*{Introduction}

Let $E$ be an elliptic curve defined over an algebraically closed field $k$ and denote by $0_E$ its identity element. Fix an integer $n>1$ not divisible by the characteristic of $k$ and consider a torsion point $P\in E(k)$ of exact order $n$. We canonically associate to it a rational function $f_P\in k(E)$ having
$$n(P)-n(0_E)$$
as divisor, inducing a degree $n$ cover
$$f_P:E\to\bold P^1$$
unramified outside $r\leq4$ points. More precisely, if the characteristic of $k$ is different from $2$, then the cover $f_P$ is generically\footnote{By generically we mean for a generic choice of the elliptic curve $E$ and of the torsion point $P$.} ramified at $4$ points: the points $0_E$ and $P$, with ramification index $n$, and two other points $R_1$ and $R_2$, with ramification index $2$. In particular, the asumption on the integer $n$ imply that the cover is tamely ramified. In some special cases, the two ramified points $R_1$ and $R_2$ coalesce, so that the resulting cover ramifies only at the points $0_E$ and $P$, always with ramification index $n$, and at a third point $Q$, with ramification index $3$, so that the cover $f_P$ is ramified above three points. This class of torsion points are called {\bf Lam\'e points} of $E$; we also say that the pair $(E,P)$ is a {\bf Lam\'e curve} of order $n$. In the algebraic group $E$, we moreover have the identity $2Q=P$; we then say that the Lam\'e curve has {\bf signature} $1$ (resp. $0$) if the point $Q$ has order $n$ (resp. $2n$).

If $k$ is of characteristic $2$ then the situation is slightly different since the cover $f_P$ is ramified at the points $0_E$ and $P$, with ramification index $n$, and at a third point $Q$, with ramification index $2$ or $3$. Here, the integer $n$ is odd, so that the cover $f_P$ is tamely ramified if and only if it has triple ramification at $Q$. We can therefore define the notion of Lam\'e points, of Lam\'e curves and of signature also in this situation (always imposing the ramification to be tame).

Let $\mathcal M_{1,2}$ denote the coarse moduli space of pairs $(E,P)$ where $E$ is an elliptic curve over $k$ and $P\neq0_E$ is a $k$-rational point. Since the existence of Lam\'e curves of given order only depends on the $k$-isomorphism class of these pairs, we can consider the subset of $\mathcal L_n(k)\subset\mathcal M_{1,2}(k)$ consisting of isomorphism classes of Lam\'e points of order $n$ and their union $\mathcal L(k)=\bigcup_n\mathcal L_n(k)$; we refer to this last set as the {\bf Lam\'e locus}. Remark that $\mathcal L_n(k)$ is contained in the modular curve $X_1(n)$, naturally considered as a closed subset of $\mathcal M_{1,2}$. It is in general very difficult to explicitly determine these sets. A first information comes from a Rigidity criterion of Weil: since the tame cover cover $f_P$ is ramified above $3$ points, the pair $(E,P)$ can be defined over the algebraic closure $k_0$ of the prime field of $k$. In particular, the set $\mathcal L(k)$ is contained in $\mathcal M_{1,2}(k_0)$, so that, when studying Lam\'e curves, we can always reduce to the case $k=\qbar$ or $k=\fpbar$.

In characteristic $0$, the Grothendieck theory of dessins d'enfants provides an elegant way enumerate $k$-isomorphism classes of Lam\'e curves, but their arithmetic behaviour still remains misterious. In positive characteristic $p>2$, it is possible to prove that a given elliptic curve admits finitely many Lam\'e points. If $k$ is of characteristic $2$ then the situation is drastically different. Indeed, we have the following Theorem, which can be considered as the main result of the paper:

\begin{theor} For any elliptic curve defined over $k=\bar{\bold F}_2$, the following conditions are equivalent:
\begin{enumerate}
\item There exists a Lam\'e point $P\in E(k)$.
\item Any element $P\in E(k)-\{0_E\}$ is a Lam\'e point.
\item The curve $E$ is supersingular.
\end{enumerate}
\end{theor}

This result is really surprising and has many applications. Here is a list of more or less direct consequences:

\begin{corol} Let $C_0$ the fiber of the canonical projection
$$\mathcal M_{1,2}\to\mathcal M_{1,1}$$
above the supersingular locus $j=0$. We then have the identity $\mathcal L(k)=C_0(k)$.
\end{corol}

We can then completely enumerate Lam\'e curves of given order and determine their field of moduli:

\begin{corol} Let $n>1$ be an odd integer. There exist exactly $\frac{n^2-1}{24}$ isomorphism classes of Lam\'e curves of order dividing $n$ over $\bar{\bold F}_2$. For any positive integer $d$, there exist $q=2^d$ isomorphism classes of Lam\'e curves whose field of moduli is contained in $\bold F_q$.
\end{corol}

We then combine these result with some previous works on this subject and then obtain the

\begin{theor} Let $(E,P)$ be a Lam\'e curve of odd order $n$ over $\qbar$. Fix an injection $\qbar\hookrightarrow\bar{\bold Q}_2$. Then $E$ has (potentially) good reduction if and only if it has signature $1$. In this case, its field of moduli (as $2$-pointed curve) is unramified above $2$.\end{theor}

We can finally obtain some global informations concerning the arithmetic of Lam\'e curves:

\begin{corol} For any number field $K$, the number of $\qbar$-isomorphism classes of Lam\'e curves of odd degree and signature $1$ defined over $K$ is bounded by $2^d$, where $d$ is the maximal degree of the residual extensions of $K$ at $2$.
\end{corol}
\newpage

\section{Lam\'e covers and Lam\'e curves}

This firs section is a brief review of known facts concerning Lam\'e covers and Lam\'e curves in arbitrary characteristic. All the results are given without proof, but many references are provided for further readings on the subject.

\subsection{Lam\'e covers and Lam\'e curves} Throughout this section, $k$ denotes an algebraically closed field of characteristic $p\geq0$. Most of the results of this paragraph can be found in~\cite{Zapponi3,Zapponi2} in the characteristic $0$ case and it is easily checked that the proofs also work in positive characteristic. A {\bf Lam\'e cover} of degree $n$ over $k$ is a tamely ramified cover
$$f:E\to\bold P^1$$ having the following ramification datum:
\begin{enumerate}
\item $f$ is of degree $n$ and unramified outside the points $\infty,0$ and $1$.
\item $f$ is totally ramified above $\infty$ and $0$.
\item There is a unique ramified point in the fiber above $1$, with ramification index $3$.
\end{enumerate}
The above properties can be summarized in a compact way by saying that $f$ has branch datum
$$(n:n:3,1,\dots,1)$$
Since such a cover is tamely ramified if and only if the characteristic $p$ does not divide the integer $3n$, we exclude the case $p=3$.

The first information comes from a classical rigidity criterion of Weil in~\cite{Weil}, which asserts that a Lam\'e cover can be defined over a finite extension of the base field of $k$. We can therefore restrict to the case $k=\qbar$ or $k=\fpbar$. This observation will be crucial when working over finite fields.

If $f$ is a Lam\'e cover of degree $n$ then, for any integer $m$ not divisible by the characteristic of $k$, the cover $f^m$ is a Lam\'e cover of order $nm$. We say that $f$ is {\bf primitive} if it cannot be factored in such a way with $m>1$.

Let $f:E\to\bold P^1_k$ be a primitive Lam\'e cover of degree $n$. It easily follows from th Riemann-Hurwitz formula that $E$ is a curve of genus $1$. We endow it with the structure of elliptic curve otained by taking the unique point $0_E\in E(k)$ lying above $\infty$ as identity element. Denote by $P$ (resp. by $Q$) the point lying above $0$ (resp. the ramified point lying above $1$). It then turns out that $P$ has exact order $n$ and we moreover find the identity $P=2Q$ (in the algebraic grou[ $E$). In the following, we refer to the pair $(E,P)$ as {\bf Lam\'e curve} of order $n$. We also say that $P$ is a {\bf Lam\'e point} of order $n$. Moreover, if the point $Q$ has order $n$ (resp. $2n$), we then say that the Lam\'e curve has {\bf signature $1$} (resp. $0$).

There exist finitely many $k$-isomorphism classes of Lam\'e curves of bounded order. It is in general difficult to determine their number and the associated field of moduli. Some results are known in characteristic $0$, essentially following from their classification via Grothendieck's theory of dessins d'enfants.

\subsection{Modular interpretation. The Lam\'e locus} Denote by $\mathcal M_{1,2}$ the moduli space of $2$-pointed elliptic curves over $k$, only considered as a coarse moduli space. Following Appendix A, we can realize it as an open subset of the weighted projective space $\bold P(1,2,3)$. An element of $\mathcal M_{1,2}(k)$ can always be represented by a $3$-ple $(E,0_E,P)$ where $E$ is an elliptic curve over $k$ and $P\in E(k)$ is different from $0_E$. Since the identity element $0_E$ is implicitly given in the definition of $E$, we can omit it and just consider the pair $(E,P)$. In particular, we can canonically associate to a Lam\'e curve $(E,P)$ an element of $\mathcal M_{1,2}(k)$. Two Lam\'e curves over $k$ are isomorphic if and only if they define the same element of $\mathcal M_{1,2}(k)$. For a fixed integer $n>1$, denote by $\mathcal L_n(k)\subset\mathcal M_{1,2}(k)$ the union of the points corresponding to Lam\'e curves of order $n$ and set
$$\mathcal L(k)=\bigcup_{n>1}\mathcal L_n(k),$$
so that $\mathcal L(k)$ (resp. $\mathcal L_n(k)$) parametrizes the $k$-isomorphism classes of Lam\'e curves over $k$ (resp. $k$-isomorphism classes of Lam\'e curves over $k$ of order $n$). The set $\mathcal L_n(k)$ is finite and in characteristic $0$ Theorem~\ref{char0} below gives an explicit expression of its cardinality. It is now important to notice that the knowledge of a Lam\'e cover is equivalent to the knowledge of the associated Lam\'e curve. In particular, the set $\mathcal L(k)$ also parametrizes $k$-isomprhism classes of Lam\'e covers over $k$. We will refer to it as the {\bf Lam\'e locus}; we similarly say that $\mathcal L_n(k)$ is the {\bf Lam\'e locus of order} $n$.

Not much is known concerning the Lam\'e locus and its canonical image in $\mathcal M_{1,1}(k)$, which parametrizes $k$-isomorphism classes of elliptic curves over $k$ admitting a Lam\'e point. In characteristic $0$, if we consider $\mathcal L(\qbar)=\mathcal L(\bold C)$ as a subset of $\mathcal M_{1,2}(\bold C)$, endowed with the usual (strong) topology, it is possible to prove (cf.~\cite{Zapponi2}) that its closure $\bar{\mathcal L}$ is connected and homeomorphic to $\mathcal M_{1,1}(\bold C)$, which suggests that the restriction of the canonical projection $\mathcal M_{1,2}(\bold C)\to\mathcal M_{1,1}(\bold C)$ to $\bar{\mathcal L}$ is a surjective homeomorphism. In positive characteristic $p\neq2$, there exist finitely many Lam\'e points on a fixed elliptic curve, their number being bounded by a constant only depending on $p$ (see for example~\cite{Zapponi6}).

\subsection{Classification in characteristic $0$.} In characteristic $0$, following the observation at the beginning of \S1.1, we can restrict to the case $k=\qbar$. Grothendieck's theory of dessins d'enfants provides an explicit and elegant enumeration of $\qbar$-isomorphism classes Lam\'e covers. Skipping the intermediate steps, we directly give the final description (see for example~\cite{Litcanu-Zapponi}). We restrict to the case of Lam\'e curves of odd degree but a similar desciption applies in the general setting. Consider the set of $3$-ples of positive integers $(a,b,c)$, where we identitfy the $3$-ples $(a,b,c),(c,a,b)$ and $(b,c,a)$ obtained by cyclic permutation of the components. The integer $n=a+b+c$ is the {\bf degree} of $(a,b,c)$. Finally, we say that $(a,b,c)$ is {\bf primitive} if $\gcd(a,b,c)=1$. Finally, the {\bf signature} of the triple $(a,b,c)$ is the image of the integer $abc$ in $\bold F_2$, which is equal to $1$ if and only if $a,b$ and $c$ are odd.

\begin{theo}\label{char0} There is a bijection between the set of $\qbar$-isomorphism classes of Lam\'e curves of odd order $n$ and signature $1$ (resp. $0$) and the set of primitive $3$-ples of degree $n$ ans signature $1$ (resp. $0$).
\end{theo}

\noindent{\bf\emph{Example.}} There exist $9$ isomorphism classes of Lam\'e curves of order $9$ over $\qbar$, associated to the triples $(1,1,7),(1,2,6),(1,6,2),(1,3,5),(1,5,3),(1,4,4),(2,2,5),(2,3,4)$ and $(2,4,3)$; only $3$ of them have signature $1$.

\subsection{Connection with Lam\'e operators with finite monodromy.} We now assume that $k=\bold C$. There is a deep link between the theory of dessins d'enfants and the theory of second order differential operators with finite monodromy. We now briefly review a special case, referring to~\cite{Beukers,Dahmen,Litcanu-Zapponi} for a more detailed and general exposition on this subjet. A {\bf Lam{\'e} operator} is a second order differential operator on the projective line defined by
$$L_n=L_{n,g_2,g_3,B}=D^2+\frac{f'}{2f}D-\frac{n(n+1)x+B}f$$
where $D=d/dx$, $f(x)=4x^3-g_2x-g_3\in\bold C[x]$ with $\Delta=g_2^3-27g_3^2\neq0$ and $B\in\bold C$. Let $E$ be the elliptic curve defined by the affine equation $y^2=f(x)$, denote by $0_E$ its origin (the point at infinity) and by $\sigma$ the canonical involution $\sigma(x,y)=(x,-y)$. We say that the operator $L_n$ is {\bf associated} to $E$ and that $B$ is the {\bf accessory parameter}. Two Lam{\'e} operators $L_{n,g_2,g_3,B}$ and $L_{n,g_2',g_3',B'}$ are {\bf equivalent} (or {\bf scalar equivalent}, following~\cite{Beukers}) if there exists $u\in\bold C^*$ such that $g_2'=u^2g_2$, $g_3'=u^3g_3$ and $B'=uB$.

Following Differential Galois Theory, the solutions (resp. the ratio of two linearly independent solutions) of the differential equation $L_n(f)=0$ generate a Galois field extension $\bold K/\bold C(t)$ (resp. $L/\bold C(t)$). The (differential) Galois group $\mbox{Gal}(K/\bold C(t))$ (resp. $\mbox{Gal}(L/\bold C(t))$) is called {\bf full monodromy} (rersp. {\bf projective monodromy}) of the Lam\'e operator. In some special cases the full monodromy is finite. Following Baldassarri's criterion in~\cite{Baldassarri}, there is a deep connection between Lam\'e curves and Lam\'e operators with finite dihedral monodromy, which can be summarized by the following result:

\begin{theo}\label{diff} There is a natural bijection between the set of equivalence classes of Lam{\'e} operators $L_1$ with finite projective dihedral monodromy of order $2n$ and the set of $\bold C$-isomorphism classes of Lam\'e curves of order $n$. Moreover, the full monodromy coincides with the projective monodromy if and only if the signature is equal to $1$, otherwise the full monodromy is dihedral of order $4n$.
\end{theo}

It follows in particular that a Lam\'e operator over $\bold C$ with finite dihedral monodromy is equivalent to a Lam\'e operator defined over $\qbar$.

\section{The characteristic $2$ case}

\subsection{Torsion points and the associated covers} Fix an elliptic curve $E$ defined over $k=\bar{\bold F}_2$, let $n>1$ be an odd integer and consider an element $P\in E[n]$ of exact order $n$. Then there exists a rational function $f_P\in k(E)$, uniquely determined up to a multiplicative constant, such that
$$(f_P)=n(P)-n(0_E).$$
This rational function induces a finite cover
$$f_P:E\to\bold P^1_k$$
The following result, which can be considered as the heart of the paper, completely descibes the ramification behaviour of this cover.

\begin{prop}\label{ramif} The cover $f_P$ is only ramified at the points $0_E$ and $P$, with ramification index $n$, and at a point $Q\in E[2n]$ such that $P=2Q$, with ramification index $2$ or $3$. More precisely, we have the following two cases:
\begin{enumerate}
\item If $E$ is ordinary then we have the identity
$$Q=\frac{n+1}2P+R,$$
where $R\in E[2]$ is the unique non-trivial $2$-torsion point, and the ramification index of $f_P$ at $Q$ is equal to $2$.
\item If $E$ is supersingular then we have the identity
$$Q=\frac{n+1}2P$$
and the ramification index at $Q$ is equal to $3$.
\end{enumerate}
\end{prop}

\begin{proof} By construction, the cover $f_P$ is totally ramified at $0_E$ and at $P$, with ramification index $n$. In particular, since $n$ is odd, the logarithmic differential form
$$\omega=\frac{df_P}{f_P}$$
has two simple poles, at $0_E$ and $P$. The zeroes of $\omega$ correspond to the ramified points of $f_P$ different from $0_E$ and $P$. Remark now that since $k$ has characteristic $2$, the order of vanishing of $\omega$ at these points is at least $2$. Now, since a canonical divisor on $E$ has degree $0$, we deduce that $\omega$ has a unique zero $Q$, of order $2$. In particular, we find the identity
$$(\omega)=2(Q)-(P)-(0_E),$$
which leads to the relation $P=2Q$ since any canonical divisor on $E$ is principal. Finally, if $f_P$ is tamely ramified at $Q$ then its ramification index is equal to $v_Q(\omega)+1=3$, otherwise it is less than or equal to $2$, and thus equal to $2$ since $f_P$ is ramified at $Q$. This proves the firs part of the proposition.

Suppose now that $E$ is ordinary and let $R\in E(k)$ be the unique non-trivial element of order $2$. The function $f_P$ is then unique if we require the extra condition $f_P(R)=1$. In this case, we find the relation
$$[-1]^*f_P=f_{-P}.$$
We start by replacing the cover $f_P$ by an isomorphic cover which is more adapted for explicit computations: set
$$P'=\frac{n+1}2P\in E[n]$$
and consider the rational function
$$g_P=\tau_{P'}^*f_P\in k(E),$$
where $\tau_{P'}$ denotes the translation by $P'$ in $E$, so that the induced cover
$$g_P:E\to\bold P^1_k$$
is isomorphic to $f_P$. By construction, we have the identity
$$(g_P)=n(P')-n(-P')$$
which leads to the relation
$$g_P=c\frac{f_{P'}}{f_{-P'}}$$
with $c\in k^\times$ and Weil's bilinear relation gives
$$c=g_P(R)=g_P(0_E)=f_P(P')=f_P(P'+R).$$
We now choose an explicit model for E: since we are only concerned with $k$-isomorphism classes of curves, we may suppose that $E$ is given by the affine Weierstrass equation
$$Y^2+XY=X^3+tX,$$
where $t^2=j(E)\in k^\times$ is the $j$-invariant of $E$. Denote by $\mathcal O$ the local ring of $E$ at the point $R=(0,0)$ and let $\frak m$ be its maximal ideal. The rational function $z=XY^{-1}\in k(E)$ is then a uniformizer at $R$ satisfying the identity
$$[-1]^*z=\frac z{z+1}.$$
Consider the truncated Taylor expansion
$$f_{P'}\equiv1+uz+vz^2\pmod{\frak m^3}$$
with $u,v\in k$. We have $u\neq0$, the contrary would imply that $f_{P'}$ is ramified at $R$, which is impossible, since $2R=0_E\neq P'$. The point $R$ being fixed by the canonical involution of $E$, we have the identity
$$[-1]^*f_{P'}\equiv1+u[-1]^*z+v[-1]^*z^2\pmod{\frak m^3},$$
from which we get
$$f_{-P'}\equiv1+uz+(u+v)z^2\pmod{\frak m^3}.$$
We then finally obtain the expansion
$$\aligned
g_P&\equiv c\frac{1+uz+vz^2}{1+uz+(u+v)z^2}\equiv c(1+uz+vz^2)(1+uz+(u^2+u+v)z^2)\equiv\\
&\equiv c(1+uz^2)\pmod{\frak m^3},
\endaligned$$
which implies that the cover $g_P$ is ramified at $R$, with ramification index $2$ or, equivalently, that $f_P$ is wildly ramified at $Q=P'+R$, as desired.

Suppose now that the curve $E$ is supersingular. Since $E(k)$ has no $2$-torsion, the point $Q$ of the proposition is the unique element of $E(k)$ for which $2Q=P$. We just have to determine the ramification index of $f_P$ at this point. We proceed exactly as in the ordinary case: first of all, we can normalize the function $f_P$ by setting $f_P(Q)=1$, so that the associated cover is unramified outside $\infty,0$ and $1$ and we obtain the identity
$$[-1]^*f_P=f_{-P}.$$
We then consider the rational function
$$g_P=\tau_{Q}^*f_P\in k(E),$$
inducing a cover isomorphic to $f_P$. By construction, we have the identity
$$(g_P)=n[Q]-n[-Q]$$
which leads to the relation
$$g_P=c\frac{f_{Q}}{f_{-Q}}$$
with $c\in k^\times$. Moreover, always by construction, the cover $g_P$ is ramified at $0_E$ and $g(0_E)=1$. We then take an affine model of $E$ by considering the affine Weierstrass equation
$$Y^2+Y=X^3.$$
Denote by $\mathcal O$ the local ring of $E$ at $0_E$ and let $\frak m$ be its maximal ideal. The rational function $z=XY^{-1}\in k(E)$ is a uniformizer at $0_E$ satisfying the identity
$$[-1]^*z=\frac z{1+Y^{-1}}.$$
Now, since $Y$ has a pole at $0_E$ of order $3$, we deduce that $Y^{-1}$ belongs to $\frak m^3$, so that we obtain the relation
$$[-1]^*z\equiv z\pmod{\frak m^3}.$$
Consider the truncated Taylor expansion
$$z^nf_{Q}\equiv u+vz+wz^2\pmod{\frak m^3}$$
with $u\in k^\times$ and $v,w\in k$. The point $0_E$ being fixed by the canonical involution of $E$, we have the identity
$$[-1]^*(z^nf_{Q})\equiv u+v[-1]^*z+w[-1]^*z^2\pmod{\frak m^3},$$
from which we easily get the relation
$$g_P\equiv1\pmod{\frak m^3}$$
which implies that the cover $g_P$ is ramified at $0_E$, with ramification index at least $3$, and thus actually equal to $3$.
\end{proof}

\subsection{The main result}

We  now translate Proposition~\ref{ramif} in terms of existence of Lam\'e points:

\begin{theo}\label{main} For any elliptic curve $E$ defined over $k=\bar{\bold F}_2$, the following conditions are equivalent:
\begin{enumerate}
\item There exists a Lam\'e point $P\in E(k)$.
\item Any element of $E(k)-\{0_E\}$ is a Lam\'e point.
\item The curve $E$ is supersingular.
\end{enumerate}
If condition 1) (or 2)) is fulfilled then the signature is always equal to $1$.
\end{theo}

\begin{proof} 1) $\Rightarrow$ 2) If $P$ is a Lam\'e point, Proposition~\ref{ramif} asserts that $E$ is supersingular. In this case, any element $P\in E(k)$ is a torsion point and its order if odd. Again from Proposition~\ref{ramif}, it is a Lam\'e point.

2) $\Rightarrow$ 3) This implication is again a direct consequence of Proposition~\ref{ramif}.

3) $\Rightarrow$ 1) Once again, from Proposition~\ref{ramif}, if $E$ is supersingular then any torsion point is a Lam\'e point.

Finally, Proposition\ref{ramif} implies that the signature is $1$, since the point $Q$ is then a multiple of $P$.
\end{proof}

In terms of the Lam\'e locus $\mathcal L(k)$ defined in the first section, consider the forgetful morphism  $\mathcal M_{1,2}\to\mathcal M_{1,1}$ and denote by $C_0$ the fiber above the point corresponding to the supersingular curve $E$, which is an affine line over $k$, isomorphic to the quotient of $E$ under the action of $\mbox{Isom}_k(E)$ (see Appendix A for further details). In this case, Theorem~\ref{main} can then restated as follows:

\begin{coro} The notation being as above, we have $\mathcal L(k)=C_0(k)$.
\end{coro}

\noindent{\bf\emph{Remark.}} It is actually possible to prove that in any positive characteristic $p$, the Zariski closure in $\mathcal M_{1,2}$ of the Lam\'e locus $\mathcal L(\bar{\bold F}_p)$ is an algebraic curve defined over $\bold F_p$. The particularity of the case $p=2$ is that this curve is 'vertical', i.e. it coincides with a fiber of the canonical projection $\mathcal M_{1,2}\to\mathcal M_{1,1}$.

\subsection{Enumeration and fields of moduli of Lam\'e curves in characteristic $2$} Let $E$ be the supersingular elliptic curve defined over $k=\bar{\bold F}_2$, given by the affine Weierstrass equation
$$Y^2+Y=X^3,$$
and consider the cover
$$\aligned
E&\stackrel\rho\longrightarrow\bold P^1_k\\
(X,Y)&\mapsto(X^4+X)^3,
\endaligned$$
which is just an explicit model of the projection $E\to E/\mbox{Isom}_k(E)$. The restriction of this cover to the open subset $E-\{0_E\}$ leads to a finite cover of the affine line. In particular, for any point $P\in E(k)-\{0_E\}$, we can consider $\rho(P)$ as an element of $k$.

\begin{theo} The morphism $\rho$ induces a bijection between the set of isomorphism classes of Lam\'e curves defined over $k$ and the set $\bold A^1(k)$ which commutes with the Galois action.\end{theo}

\begin{proof} It is a straightforward consequence of Theorem~\ref{main}.
\end{proof}

Consider the multiplicative function $\psi:\bold N\to\bold N$ defined by
$$\psi(p^r)=p^{2r-2}(p^2-1)$$
for any odd prime number $p$ and any positive integer $r$. Remark that for $n$ odd, $\psi(n)$ is just the number of elemets of $E(k)$ of exact order $n$. The next corollaries direct follow from the above result.

\begin{coro}\label{enu2} For any odd integer $n>1$ not divisible by $3$, there exist $\frac{n^2-1}{24}$ (resp. $\frac{\psi(n)}{24}$) isomorphism classes of Lam\'e curves of order dividing $n$ (resp. of exact order $n$) over $k$.
\end{coro}

\noindent{\bf\emph{Remark.}} For simplicity, the above result just describes the case where $n$ is prime to $3$. If $n=3m$, we obtain $\frac{3m^2+5}8$ isomorphism classes of Lam\'e curves of order dividing $n$.

Let now $\eta:\bold N\to\bold N$ be the multiplicative function defined by
$$\eta(p^r)=2^{p^r}-2^{p^{r-1}}$$
for any prime number $p$ and any positive integer $r$.

\begin{coro} For any positive integer $d$, there exist exactly $q=2^d$ (resp. $\eta(d)$) isomorphism classes of Lam\'e curves over $k$ whose field of moduli is contained in (resp. coincides with) $\bold F_q$.
\end{coro}

\subsection{Lifting from and reducing to characteristic $2$} From now on, we fix an injection $\qbar\hookrightarrow\bar{\bold Q}_2$, so that we can consider semi-stable models and define the notion of (potentially) good reduction of Lam\'e curves. Here, by good reduction, we just mean that the elliptic curve $E$ has potentially good reduction, which is equivalent, following Deuring's Theorem, to the fact that the $j$-invariant of $E$ has non-negative $2$-valuation.

\begin{theo} A Lam\'e curve of odd degree over $\bar{\bold Q}_2$ has potentially good reduction if and only if it has signature $1$. In particular, there is a bijection between the set of Lam\'e curves of odd degree and signature $1$ over $\bar{\bold Q}_2$ and the set of Lam\'e curves over $\bar{\bold F}_2$.
\end{theo}

\begin{proof} Let $(E,P)$ be a Lam\'e curve of odd degree over $\bar{\bold Q}_2$ and denote by $f:E\to\bold P^1$ the corresponding Lam\'e cover (unramified above $\infty,0$ and $1$. The set
$$S=\{R\in E(k)\,\,|\,\, 2R=P\}$$
has order $4$ and only one of its elements has order $n$. By construction, the ramified point of index $3$ belongs to $S$ and lies above $1$ under $f$. Following the results of \S2.4 in~\cite{Zapponi3} (where the notation is slightly different, the origin of $E$ being the ramified point of order $3$) an element of $S$ lies above $1$ or $-1$ and the signature is equal to $0$ if and only if exactly $3$ points of $S$ lie above $1$ (this result may also be derived from Weil's reciprocity law). Consider the semi-stable model $\mathcal E\to\mathcal X$ of $f$ and suppose that $E$ has good reduction. In general, the reduced curve $\bar{\mathcal E}$ may not have good reduction, but the reduced curve $\bar E$ appears as one of its irreducible components. Moreover, it follows from the particularity of the covers studied here that the restriction of the morphism  $\bar{\mathcal E}\to\bar{\mathcal X}$ do $\bar E$ coincides with the canonical cover $f_{\bar P}$ associated to $(\bar E,\bar P)$, as defined in \S2.1. Now, if the signature is $0$, then the ramification index of the ramified point above $1$ would be at least $4$ (since the reduction of $S$ has cardinality at most $2$), which is impossible from Proposition~\ref{ramif}. It then follows that $E$ has bad reduction.

Alternatively, Theorem 5.1 in~\cite{Zapponi5} (where, once again, the notation is slightly different) gives a full enumeration of the isomorphism classes of Lam\'e curves of given order with bad reduction at a given prime. For $p=2$, in the case of odd degree, we recover exactly the set of Lam\'e curves of signature $0$.

Finally, following Theorem~\ref{char0} and Corollary~\ref{enu2}, it is easily checked that the number of isomorphism classes of Lam\'e curves over $\bar{\bold Q}_2$ of odd degree and signature $1$ coincides with the number of isomorphism classes of Lam\'e curves over $\bar{\bold F}_2$. The result follows from the fact that an isomorphism class of tamely ramified covers over $\bar{\bold F}_2$ uniquely lifts to an isomorphism class of covers over $\bar{\bold Q}_2$ having the same ramification datum (see for example~\cite{sga1}).
\end{proof}

\noindent{\bf\emph{Remark.}} It follows moreover that a Lam\'e curve of odd degree has good reduction if and only if the same holds for the associated Lam\'e cover (which is generally a stronger condition).

\ 

Let $\bold Q_2^{\tiny ur}\subset\bar{\bold Q}_2$ be the maximal unramified extension and, for any positive integer $d$, let $K_d$ the unique unramified degree $d$ extension of $\bold Q_2$. Since the lifting from characteristic $2$ commutes with the Galois action, we obtain the following result:

\begin{coro} Any Lam\'e curve of odd degree and signature $1$ can be defined over $\bold Q_2^{\tiny ur}$. More precisely, there exist exactly $2^d$ (resp. $\eta(d)$) isomorphism classes of such curves whose field of moduli is contained in (resp. coincides with) $K_d$.
\end{coro}

\begin{proof} This is a classical argument: first of all, the unicity of the lift implies that the inertia acts trivially on the isomorphism classes of Lam\'e curves of odd degree and signature $0$. In other worlds, the field of moduli of a Lam\'e curve of odd degree and signature $1$ is unramified. The result then follows from the fact that in this special situation, the field of moduli is the minimal field of definition (this is clear for $n>3$, since in this case Lam\'e curves have trivial automorphism group and the case $n=3$ can be explicitly treated).
\end{proof}

\noindent{\bf\emph{Remark.}} In the case of signature $0$, since the curve has bad reduction, the local field of moduli is completely determined in~\cite{Zapponi5}.

\ 

Finally, these some local informations directly lead to the following straightforward result in the global setting:

\begin{coro} Any Lam\'e curve of odd degree and signature $1$ over $\bar{\bold Q}$ can be defined over a number field unramified above $2$ and has good reduction at any prime lying above $2$. Furthermore, the number of Lam\'e curves of odd degree and signature $1$ defined over a fixed number field $K$ can be bounded by a constant only depending on the (maximal) degree of the residual extensions of $K$ at $2$.
\end{coro}

\section*{Appendix A - An explicit construction of the moduli space $\mathcal M_{1,2}$}

This paragraph with a brief review of classical results concerning the moduli spaces $\mathcal M_{1,1}$ and $\mathcal M_{1,2}$. As a coarse moduli space over $k$, we have a natural isomorphism
$$\mathcal M_{1,1}\cong\bold A^1_k,$$
explicitly given by associating to an elliptic curve $E$ its $j$-invariant $j(E)\in k$. The moduli space $\mathcal M_{1,2}$ can be realized as the complement in the weighted projective space
$$\bold P(1,2,3)=\mbox{\bf Proj}(k[a,b,c]),$$
(where $a,b$ and $c$ are homogeneous of restective degree $1,2$ and $3$) of the divisor $\Delta$ defined by the homogeneous equation
$$-c^2(ba^4+8a^2b^2+16b^3-a^3c+27c^2-36abc)=0.$$
More precisely, the elliptic curve associated to an element $[a,b,c]\in\bold P(1,2,3)(k)$ is explicitly given by the affine Weierstra\ss\, equation
$$E_{a,b,c}: Y^2+aXY+cY=X^3+bX^2,$$
the marked points being the origin $0_E$ (the point at infinity) and the point $(0,0)$. The forgetful morphism $\mathcal M_{1,2}\to\mathcal M_{1,1}$ simply maps the element $[a,b,c]$ to the $j$-invariant of the curve $E_{a,b,c}$, given by
$$j(E_{a,b,c})=-\frac {(16b^2+8ba^2+a^4-24ac)^3}{c^2(ba^4+8a^2b^2+16b^3-a^3c+27c^2-36abc)}.$$
For any elliptic curve $E$, there is a canonical morphism
$$\rho:E-\{0_E\}\to\mathcal M_{1,2}$$
whose image is isomorphic to the quotient of $E-\{0_E\}$ under the action of $\mbox{Isom}_k(E)$ and coincides with the fiber $C_j$ above $j=j(E)$ of the natural projection
$$\mathcal M_{1,2}\to\mathcal M_{1,1}.$$
Finally, it is esaily checked that the field of definition of the element $[a,b,c]\in\bold P(1,2,3)(k)$ coincides with the field of moduli of the corresponding $2$-pointed elliptic curve and that it is in fact a field of definition for the curve.

\section*{Appendix B - Higher genus analogues}

\subsection*{B.1. Generalized Lam\'e covers and Lam\'e curves} We now introduce a class of covers of the projective line which generalizes Lam\'e covers. As in \S1, let $k$ denote an algebraically closed field of characteristic $p\geq0$. Fix two positive integers $g$ and $n$ such that $2g+1\leq n$. A {\bf generalized Lam\'e cover} of degree $n$ over $k$ is a tamely ramified cover
$$f:C\to\bold P^1$$ having the following ramification datum:
\begin{enumerate}
\item $f$ is of degree $n$ and unramified outside the points $\infty,0$ and $1$.
\item $f$ is totally ramified above $\infty$ and $0$.
\item There is a unique ramified point in the fiber above $1$, with ramification index $2g+1$. 
\end{enumerate}
The above properties can be summarized in a compact way by saying that $f$ has branch datum
$$(n:n:2g+1,1,\dots,1)$$
Since such a cover is tamely ramified if and only if the characteristic $p$ does not divide the integer $(2g+1)n$, we exclude the case where $p$ divides $2g+1$. For $g=0$, we find the usual cyclic cover $\bold P^1_k\to\bold P^1_k$ unramified outside $2$ points. For $g=1$, we obtain Lam\'e covers, as defined previously. In general, the Riemann-Hurwitz formula implies that $C$ is a curve of genus $g$. Once again, the rigidity criterion of Weil asserts that a generalized Lam\'e cover can be defined over a finite extension of the base field of $k$ and, furhtermore, there exist finitely many $k$-isomorphism classes of Lam\'e covers of bounded degree.

Given a generalized Lam\'e cover $f:C\to\bold P^1_k$ of degree $n$ and genus $g$, denote by $P_\infty$ (resp. by $P_0$) the point lying above $\infty$ (resp. above $0$). Let moreover $P_1$ be the only ramified point lying above $1$. In the following, we say that the $3$-ple $(C,P_\infty,P_0)$ is a {\bf generalized Lam\'e curve of order $n$}. Remark that, as in the case of genus $1$, the cover $f$ (and in particular the point $P_1$) is completely determined by the corresponding generalized Lam\'e curve. By construction, the divisor
$$D=(P_0)-(P_\infty)$$
defines a point of the Jacobian variety $J(C)$ of order dividing $n$. We can moreover consider the canonical image of a generalized Lam\'e curve in the moduli space $\mathcal M_{g,2}(k)$ and define the {\bf generalized Lam\'e loci} $\mathcal L_{g,n}(k)$ and $\mathcal L_g(k)$.

\subsection*{B.2. A class of hyperelliptic curves} We now suppose that $k=\bar{\bold F}_2$. For any positive integer $g$, let $C_g$ be the genus $g$ hyperelliptic curve defined by the affine equation
$$Y^2-Y=X^{2g+1}.$$
Denote by $\infty\in C_g(k)$ the (unique) point at infinity. Let moreover $\sigma\in\mbox{Aut}_k(C_g)$ the canonical involution, defined by
$$\sigma(X,Y)=(X,Y+1).$$
Remark that $\infty$ is the only point of $C_g$ fixed by $\sigma$. Moreover, it is easily checked that the curve $C_g$ is supersingular. We can now state an analogue (but weaker version) of Theorem~\ref{main}; we omit its proof, being almost the same than in genus $1$.

\begin{theoc} For any point $P\in C(k)-\{\infty\}$, the $2$-pointed curve $(C_g,P,\sigma(P))$ is a generalized Lam\'e curve.
\end{theoc}

Denote by  $X_g\cong\bold P^1_k$ be the quotient of $C_g$ under the action of the stabilizer of the point $\infty$ in $\mbox{Aut}_k(C_g)$.

\begin{coroc} The set $X_g(k)$ naturally parametrizes $k$-isomorphism classes of generalized Lam\'e covers.\end{coroc}

The main inconvenient of this result is the fact that there may exist generalized Lam\'e covers which do not arise in this way (which is never the case for $g=1$ ). A more detailed description of the moduli space $\mathcal M_{g,2}$ would lead to some more precise statements.

\bibliographystyle{smfalpha}
\bibliography{Belyidegree}

\providecommand{\bysame}{\leavevmode ---\ }
\providecommand{\og}{``}
\providecommand{\fg}{''}
\providecommand{\smfandname}{et}
\providecommand{\smfedsname}{\'eds.}
\providecommand{\smfedname}{\'ed.}
\providecommand{\smfmastersthesisname}{M\'emoire}
\providecommand{\smfphdthesisname}{Th\`ese}
\begin{thebibliography}{BvdW04}

\bibitem[Bal87]{Baldassarri}
{\scshape F.~Baldassarri} -- {\og Algebraic solutions of the {L}am\'e equation
  and torsion of elliptic curves\fg}, \emph{Proceedings of the {G}eometry
  {C}onference ({M}ilan and {G}argnano, 1987)}, vol.~57, 1987, p.~203--213
  (1989).

\bibitem[BvdW04]{Beukers}
{\scshape F.~Beukers {\normalfont \smfandname} A.~van~der Waall} -- {\og Lam\'e
  equations with algebraic solutions\fg}, \emph{J. Differential Equations}
  \textbf{197} (2004), no.~1, p.~1--25.

\bibitem[Dah07]{Dahmen}
{\scshape S.~R. Dahmen} -- {\og Counting integral {L}am\'e equations by means
  of dessins d'enfants\fg}, \emph{Trans. Amer. Math. Soc.} \textbf{359} (2007),
  no.~2, p.~909--922 (electronic).

\bibitem[LZ06]{Litcanu-Zapponi}
{\scshape R.~Li{\c{t}}canu {\normalfont \smfandname} L.~Zapponi} -- {\og
  Properties of {L}am\'e operators with finite monodromy\fg}, Groupes de
  {G}alois arithm\'etiques et diff\'erentiels, S\'emin. Congr., vol.~13, Soc.
  Math. France, Paris, 2006, p.~235--252.

\bibitem[sga03]{sga1}
\emph{Rev\^etements \'etales et groupe fondamental ({SGA} 1)} -- Documents
  Math\'ematiques (Paris) [Mathematical Documents (Paris)], 3, Soci\'et\'e
  Math\'ematique de France, Paris, 2003, S{\'e}minaire de g{\'e}om{\'e}trie
  alg{\'e}brique du Bois Marie 1960--61. [Algebraic Geometry Seminar of Bois
  Marie 1960-61], Directed by A. Grothendieck, With two papers by M. Raynaud,
  Updated and annotated reprint of the 1971 original [Lecture Notes in Math.,
  224, Springer, Berlin; MR0354651 (50 \#7129)].

\bibitem[Wei56]{Weil}
{\scshape A.~Weil} -- {\og The field of definition of a variety\fg},
  \emph{Amer. J. Math.} \textbf{78} (1956), p.~509--524.

\bibitem[Zap97]{Zapponi3}
{\scshape L.~Zapponi} -- {\og Dessins d'enfants en genre 1\fg}, Geometric
  {G}alois actions, 2, London Math. Soc. Lecture Note Ser., vol. 243, Cambridge
  Univ. Press, Cambridge, 1997, p.~79--116.

\bibitem[Zap98]{Zapponi2}
\bysame , \emph{Dessins d'enfants et action galoisienne}, Th\`ese de doctorat,
  Universit\'e de Franche-Comt\'e, Besan\c con, 1998.

\bibitem[Zap06a]{Zapponi5}
\bysame , \emph{Lam\'e curves with bad reduction}, 2006, Preprint aviable at
  the address \url{http://arxiv.org/pdf/math.AG/0611429.pdf}.

\bibitem[Zap06b]{Zapponi6}
\bysame , \emph{Lam\'e points on elliptic curves}, 2006, Preprint aviable at
  the address
  \url{http://people.math.jussieu.fr/\~zapponi/leo_divers/Lame_points.pdf}.

\end{thebibliography}

\end{document}